 \documentclass[preprint,11pt]{elsarticle}

\usepackage{amsmath}
\usepackage{amssymb}
\usepackage{amsthm}

\newtheorem{Th}{Theorem}[section]

\newtheorem{Lemma}[Th]{Lemma}
\newtheorem{Problem}[Th]{Problem}
\newtheorem*{A}{Casas-Alvero Conjecture}
\theoremstyle{remark}
\newtheorem*{sol}{Solution}

\newcommand{\dst}{\displaystyle}
\begin{document}

\journal{J. Approx. Theory}

\begin{frontmatter}



\title{Three proofs of the Casas-Alvero conjecture}

\author[poli]{Luis J. Fern\'andez de las Heras\corref{autor}}
\ead{lfernandez@etsii.upm.es}
\author[acu]{Mar\'ia J. Fern\'andez de las Heras}
\ead{mjose@caend.upm-csic.es}

 \cortext[autor]{Corresponding author.
Fax: +34 91 336 3001. Phone: +34 91 336 3106}

\address[poli]{Dpto. de Matem\'atica Aplicada,
E.T.S.  de Ingenieros Industriales, Universidad Polit\'ecnica de
Madrid, Jos\'e G. Abascal 2, 28006 Madrid, Spain.}

\address[acu]{CSIC,
      Centro de Ac\'ustica Aplicada y Evaluaci\'on no Destructiva,
      Serrano 144,
  28006 Madrid, Spain.}

\begin{abstract}
The Casas-Alvero conjecture claims that a complex univariate
polynomial having roots in common  with each of its derivatives must
be a power of a linear polynomial. Up to now, only partial proofs
and numerical evidences have been presented. In this paper we give
three different proofs of the conjecture.
\end{abstract}

\end{frontmatter}

\section{Introduction}

This paper is concerned with the following question posed by E.
Casas-Alvero more than a decade ago.
\begin{A}
 Let  $f$ be a monic complex polynomial of degree n in a
single variable $z$. Suppose that $\gcd(f, f^{(k)} )\not= 1$ for $k
= 1, \dots , n-1$, where $f^{(k)}$ denotes the $k$-th derivative of
$f$. Then, there exists a constant $a\in\mathbb{C}$ such that $f(z)
= (z-a)^n$.
\end{A}
It may be proven that if the conjecture is true over $\mathbb{C}$
then it is true over all fields of characteristic $0$. In contrast,
the conjecture is not true in prime characteristic.

In \cite {2} the Casas-Alvero conjecture was proven for polynomials
of degree $n$ less than or equal to $7$. The conjecture has also
been proven for infinitely many values of $n$, see \cite{1}.

We can rewrite the conjecture in terms of interpolation polynomials
on the complex plane $\mathbb{C}$, see \cite{3,4}.
\begin{Th}
\label{pint} Let $z_1 , z_2 , \dots, z_n $ be $n$ complex numbers
and let $p$ be a monic complex polynomial of degree $n$. Suppose
that the polynomial $p$ satisfies
$$\begin{array}{rcl}
p (z_k) &=& 0,\qquad               k = 1, \dots, n;\\
p^{(k)} (z_{k+1}) &=& 0,\qquad      k = 1, \dots, n-1.
\end{array}
$$
Then
\[
p (z) = (z-a)^n,\qquad a\in\mathbb{C}.
\]
\end{Th}

In the sequel any monic polynomial satisfying the conditions of
Theorem \ref{pint} will be called a Casas-Alvero polynomial.

\section{Birkhoff interpolation}
First, we pose the following problems which may be considered as
particular cases of Birkhoff  or lacunary  interpolation problems
\cite{5}. They are closely related to the Casas-Alvero conjecture.

\begin{Problem}
\label{p1} Given $n$ complex numbers $ \alpha_1,\dots \alpha_n$,
find all monic polynomials $p$ that satisfy
\[
p^{(k)} (\alpha_{k+1}) = 0,\qquad k = 0, 1,\dots, n-1.
\]
\end{Problem}

\begin{sol}
A monic polynomial $p_n$ of degree $n$ can be expressed in the form
\begin{equation}\label{monico}
p_n (z) = \sum_{j=0}^n a_j z^j,
\end{equation}
where $a_n = 1$. For $k=1,\dots,n$, its $k$-th derivative is given
by
\[
p_n^{(k)}(z)=\sum_{j=k}^n a_j \frac{j!}{(j-k)!}z^{j-k}.
\]
If we evaluate the above expressions at the nodes, we obtain
\[
p_n^{(k)} (\alpha_{k+1}) = \sum_{j=k}^n a_j
\frac{j!}{(j-k)!}\,\alpha_{k+1}^{j-k}=0,\quad k=0,1,\dots,n-1.
\]
We can write the above equations in matrix form as
$$
A\cdot \left(\begin{array}{c}
a_0\\ a_1\\ \vdots\\a_{n-1}\\a_n\\
\end{array}\right)=\left(\begin{array}{c}
0\\0\\ \vdots\\0\\n!\\
\end{array}\right),
$$
where

\begin{eqnarray*}
A=\left(\begin{array}{ccccccc} 1 & \alpha_1 & \alpha_1^2
&\cdots&\alpha_1^{n-2}&
\alpha_1^{n-1}&\alpha_1^n\\
0& 1 & 2\alpha_2 &\cdots&
\frac{n-2}{1}\,\alpha_2^{n-3}&\frac{n-1}{1}\,\alpha_2^{n-2}&n \alpha_2^{n-1}\\
0 & 0 &2! &\cdots&
\frac{(n-2)!}{(n-4)!}\,\alpha_3^{n-4}&\frac{(n-1)!}{(n-3)!}\,
\alpha_3^{n-3}& \frac{n!}{(n-2)!}\,\alpha_3^{n-2}\\
\vdots &\vdots &\vdots &\cdots&\vdots&\vdots&\vdots \\
0 &0 &0
&\cdots&(n-2)!&\frac{(n-1)!}{1}\,\alpha_{n-1}&\frac{n!}{2!}\,
\alpha_{n-1}^2 \\
0 &0 &0 &\cdots&0&(n-1)! &n!\alpha_n\\
0 &0 &0 &\cdots&0&0 &n!\\\
\end{array}\right).
\end{eqnarray*}

Clearly the interpolation problem has a solution and it is unique,
because the matrix is of full rank. The solution of the system is
trivial but laborious. The incidence matrix will have as many rows
as the number of different nodes. In the extreme case that there is
only one node, the problem is of Hermite type. As in all cases the
matrix is triangular, there is always a unique solution.

A straightforward calculation shows that the unique solution $p_n$
of Problem \ref{p1} admits the following integral representation
\begin{eqnarray}\label{int}
p_n (z)= n! \int_{\alpha_1}^z dx_1 \int_{\alpha_2}^{x_1} dx_2
 \dots\int_{\alpha_{n-1}}^{x_{n-2}}
dx_{n-1} \int_{\alpha_n}^{x_{n-1}} dx_n.
\end{eqnarray}
\end{sol}

We can study now the inverse problem.

\begin{Problem}
\label{p2} Given $n$ complex numbers $z_1 , z_2 , \dots , z_n$,
 let
\[
\hat{z}=(z_1,\dots,z_n)\quad \mbox{and}\quad p(z)=\prod_{k=1}^{n}
(z-z_{k}).
\]
Find all the vectors of complex numbers
\[
\alpha =(\alpha_1 , \alpha_2 , ... ,\alpha_n )
\]
that satisfy
\[
p^{(k)} (\alpha_{k+1}) = 0,\qquad k = 0, 1, ..., n-1.
\]
\end{Problem}

\begin{sol}
The problem has a solution. We begin with the last equation of the
system given by Problem \ref{p1} to obtain $\alpha_n$. Then we solve
the $(n-1)$-th equation to get the two values of $\alpha_{n-1}$ and
we continue until considering the first equation of degree $n$ which
allows us to obtain the $n$ values of $\alpha_1$.

The problem has a solution but the solution is not unique. The
problem may have until $n!$ different solutions. Of course, it is
possible to approximate the solutions by numerical methods, but it
should be noticed that this problem is invariant under permutation
of the symmetric group of the $n$ components of $\hat{z}$.
 If we change the order of the roots in $\hat{z}$, the polynomial $p$
 does not change and the solutions of the problem are the same ones as before. The
number of solutions is just equal to the number of permutations of
$\hat{z}$.

We can make Problem \ref{p2} more complicated considering additional
 conditions like requiring the
nodes of interpolation  $\alpha$ to be  equal to the zeros of the
polynomial $p$. Thus, the polynomial $p$ and its derivatives will
have common zeros.
\end{sol}

\section{The first proof}
 There are several different ways to tackle the
proof. We have chosen to consider it as a problem of the same type
as that of Problem \ref{p2}. That is, assuming that the roots of the
polynomial are known, find the nodes of interpolation. Subsequently,
impose the condition that the roots coincide with the nodes. Then
this problem is invariant under permutations of the zeros of the
polynomial as in Problem \ref{p2}. But if we permute the zeros of
the polynomial, the interpolation nodes are automatically
interchanged. Thus the polynomial and its derivatives up to order
$n-1$ must take the value zero at all the zeros of the polynomial.
Therefore, all the zeros are of multiplicity $n$.
 But this is possible only if there is a single zero of multiplicity $n$.
 Thus, the Casas-Alvero conjecture is proven.

\section{A second proof by induction}
 Taking account of the representation \eqref{int},
for $n=2$, the Casas-Alvero polynomial is
$$
\begin{array}{rcl}
p_2 (z)&=&\dst 2! \int_{\alpha_1}^z dx_1 \int_{\alpha_2}^{x_1} dx_2
= 2! \int_{\alpha_2}^z dx_1 \int_{\alpha_2}^{x_1} dx_2\\ \\&=&\dst (
z-\alpha_2)^2-(\alpha_1-\alpha_2)^2=( z-\alpha_2)^2=( z-\alpha_1)^2.
\end{array}
$$
Now suppose that, for $k=1,\dots,n-1$, any Casas-Alvero polynomial
is of the form
\begin{eqnarray*}
p_k (z)=( z-\alpha_1)^k.
\end{eqnarray*}
Then
\begin{eqnarray*}
p_{n} (z)=n! \int_{\alpha_1}^z dx_1 \int_{\alpha_2}^{x_1} dx_2
\cdots \int_{\alpha_{n-1}}^{x_{n-2}} dx_{n-1}
\int_{\alpha_n}^{x_{n-1}} dx_n.
\end{eqnarray*}
Notice that $p_n (z)$ takes the value zero at $z=\alpha_k$ from
$k=1$ to $n$. So, for $k=1,\dots,n$, we have
\begin{eqnarray*}
p_n (z)=n! \int_{\alpha_k}^z dx_1 \int_{\alpha_2}^{x_1} dx_2
\cdots\int_{\alpha_{n-1}}^{x_{n-2}} dx_{n-1}
\int_{\alpha_n}^{x_{n-1}} dx_n.
\end{eqnarray*}
Therefore
\begin{eqnarray*}
p_n (z)&=&n! \int_{\alpha_k}^z dx_1 \int_{\alpha_2}^{x_1} dx_2 \cdots\int_{\alpha_{n-1}}^{x_{n-2}} (x_{n-1}-{\alpha_n}) dx_{n-1}\\
&=&n! \int_{\alpha_k}^z dx_1 \int_{\alpha_2}^{x_1} dx_2
\cdots\int_{\alpha_{n-1}}^{x_{n-2}} x_{n-1} dx_{n-1}-n\alpha_n
p_{n-1}(z),
\end{eqnarray*}
for $k=1,\dots,n$, where  $p_{n-1}$ is a Casas-Alvero polynomial of
degree $n-1$. Then
 \begin{eqnarray*}  p_{n-1}(z) =( z-\alpha_1)^{n-1}
\end{eqnarray*}
and
$$
\alpha_1=\alpha_2=\dots=\alpha_{n-1}.
$$
Therefore
\begin{eqnarray*}
p_n(z)&=&n! \int_{\alpha_1}^z dx_1 \int_{\alpha_1}^{x_1} dx_2 \cdots
\int_{\alpha_1}^{x_{n-3}}\frac{1}{2} (x_{n-2}^2
-{\alpha_1^2})\,dx_{n-2}-n\alpha_n p_{n-1}(z)\\
&=&n!\int_{\alpha_1}^z dx_1 \int_{\alpha_1}^{x_1} dx_2
\cdots\int_{\alpha_1}^{x_{n-3}}\frac{1}{2!}
x_{n-2}^2\,dx_{n-2}\\
&&-\frac{n(n-1)}{2}\alpha_1^2 p_{n-2}(z)-n\alpha_n p_{n-1}(z),
\end{eqnarray*}
where $  p_{n-2}(z) =( z-\alpha_1)^{n-2}$.

If we continue the process of calculating the iterated integral, we
arrived at
\begin{eqnarray*}
p_n (z)&=&z^n-\alpha_1^n-\sum_{k=1}^{n-2} \left(\begin{array}{c}n\\k\\
\end{array}\right)\alpha_1^{n-k} (z-\alpha_1)^k-n\alpha_n p_{n-1}(z)\\
&=&z^n-\sum_{k=0}^{n-1} \left(\begin{array}{c}n\\k\\
 \end{array}\right)
\alpha_1^{n-k} (z-\alpha_1)^k+n(\alpha_1-\alpha_n)(z-\alpha_1)^{n-1}\\
&=&z^n+(z-\alpha_1)^n-\sum_{k=0}^n \left(\begin{array}{c}n\\k\\
 \end{array}\right)
\alpha_1^{n-k} (z-\alpha_1)^k\\ && +n(\alpha_1-\alpha_n)(z-\alpha_1)^{n-1}\\
&=&(z-\alpha_1)^n+n(\alpha_1-\alpha_n)(z-\alpha_1)^{n-1},
\end{eqnarray*}
since the number $\alpha_n$ is a root of the polynomial $p_n$.
Actually
\begin{eqnarray*}
p_n (\alpha_n)&=&(\alpha_n-\alpha_1)^n+n(\alpha_1-\alpha_n) (\alpha_n-\alpha_1)^{n-1}\\
&=&-(n-1)(\alpha_n-\alpha_1)^n=0.
\end{eqnarray*}
Then, $\alpha_n=\alpha_1$ and $p_n(z)=(z-\alpha_1)^n$, as we wanted
to prove.

\section{The third proof}

Although we think that the previous two proofs are correct, we
encourage the reader to continue reading the article.

\begin{Lemma}
\label{lema} Every Casas-Alvero  polynomial has a root equal to the
geometric center of gravity of its roots.
\end{Lemma}
\begin{proof} Let $p_n$  be a monic polynomial of degree $n$
expressed in the form \eqref{monico}. Then
$$
p_n^{(n-1)} (z)= n!\,z+(n-1)!\,a_{n-1}.
$$
Evaluating at the zero $z_n$ of the polynomial $p_n^{(n-1)}$, we
obtain
$$ p_n^{(n-1)} (z_n)= n!z_n+(n-1)!a_{n-1}=0,
$$
or, equivalently,
$$
z_n =-\frac{a_{n-1}}{n} =\frac{1}{n}\sum_{j=1}^{n}z_j.
$$
\end{proof}

Now, let $p_n$  be a monic polynomial of degree $n$. The polynomial
$p_n$ may be written in the form
\begin{eqnarray}\label{e1}
p_n (z)=\sum_{k=0}^{n}\left(\begin{array}{c}n\\k\\
\end{array}\right) (-1)^k c_k z^{n-k},
\end{eqnarray}
where
\begin{eqnarray*}
 c_0 &=&  1,\qquad c_1 ={\left(\begin{array}{c}n\\1\\
\end{array}\right)}^{-1} \sum_{j=1}^n  z_j, \\ \dst
 c_k&=&  {\left(\begin{array}{c}n\\k\\
\end{array}\right)}^{-1} \sum_{ j_1+j_2+...+j_n=k}^{j_i=0,1}
\left( \prod_{i=1} ^n  z_i^{j_i}\right),\quad k=2,\dots, n.
\end{eqnarray*}

Taking derivatives in \eqref{e1}, we obtain
\begin{eqnarray*}
 p'_n (z)&=&\sum_{k=0}^{n-1}\left(\begin{array}{c}n\\k\\
\end{array}\right) (-1)^k c_k (n-k) z^{n-k-1}\\
&=&n\sum_{k=0}^{n-1}\left(\begin{array}{c}n\\k\\
\end{array}\right) (-1)^k c_k \frac{n-k}{n} z^{n-k-1}\\
&=&n\sum_{k=0}^{n-1}\left(\begin{array}{c}{n-1}\\k\\
\end{array}\right) (-1)^k c_k  z^{n-k-1}.
\end{eqnarray*}

The above calculation indicates that the zeros of the derivative of
a polynomial of degree $n$ and the zeros of this polynomial have in
common the following quantities: the average of their zeros, the
mean double product of their zeros, and so forth until the average
$(n -1)$ product of their zeros.

If $p_n$ is a Casas-Alvero polynomial, then, for $k=1,\dots,n-1$,
the polynomial $p_n^{(k)}$ has the same coefficients $c_j,\,
j=1,\dots,n-k,$ as $p_n$.

For $k = n-1$, the zero shared by the polynomial and its derivative
of order $n-1$ is precisely $c_1$ and this determines the
coefficient of degree $n-1$ of the polynomial $p_n$.

For $k = n-2$, which is the zero shared by the polynomial and the
 derivative of order $n-2$? The two zeros determine the same value for $c_2$.

Therefore, whatever zero $p_n$ shares with its derivative we obtain
same result for the polynomial. But this reasoning is valid for all
the derivatives of $p_n$. Then all the zeros of the successive
derivatives of $p_n$ must be zeros of $p_n$, therefore all the zeros
must be equal.

\section{Concluding remarks}

We conclude that the Casas-Alvero conjecture is true. In connection
with the work carried out in this paper, we are currently studying
an optimization problem in which symmetries play an important role.
The results related to this problem will appear elsewhere.

In our opinion solving a math problem is not to close the door but
rather open a window to new challenges. For this reason, we propose
a degenerated Birkhoff interpolation problem in which the number of
equations is greater than the number of unknowns. Among the
different possibilities we have chosen the problem stated below.
Find the conditions to be met by the interpolation nodes and
 their values for the following problem to be solvable.
\begin{Problem} Given $2n$ complex numbers $
\alpha_1,\dots,\alpha_n $ and $ c_0,\dots,c_{n-1} $, find all the
polynomials $p$ of degree $n$ that satisfy
\[
p (\alpha_{k+1})=p^{(k)} (\alpha_{k+1}) = c_k, \qquad k = 0, 1,
\dots, n-1.
\]
\end{Problem}


\end{document}